\documentclass[11pt]{article}
\usepackage{amsmath,amssymb,amsthm,amscd,esint}

\newtheorem{theorem}{Theorem}[section]
\newtheorem{lemma}[theorem]{Lemma}

\newtheorem{corollary}[theorem]{Corollary}

\newtheorem{conjecture}[theorem]{Conjecture}

\theoremstyle{definition}
\newtheorem{definition}[theorem]{Definition}

\theoremstyle{remark}
\newtheorem{remark}[theorem]{Remark}

\numberwithin{equation}{section}
\numberwithin{theorem}{section}

\DeclareMathOperator{\vol}{vol}

\DeclareMathOperator{\diam}{diam}

\DeclareMathOperator{\Ric}{Ric}

\DeclareMathOperator{\KE}{KE} \DeclareMathOperator{\GKE}{GKE}
\DeclareMathOperator{\reg}{reg} \DeclareMathOperator{\WP}{WP}

\begin{document}

\title{Relative volume comparison of Ricci Flow and its applications}

\author{
Gang Tian\thanks{Supported partially by
NSF grants DMS-1309359, DMS-1607091 and NSFC grant 11331001. Email: tian@math.princeton.edu}\\
Beijing University
\\[5pt]
Zhenlei Zhang\thanks{Supported partially by NSFC grants 11431009, 11771301 and Support Project of High-level Teachers in Beijing. Email: zhleigo@aliyun.com}\\
Capital Normal University}
\date{}

\maketitle

\begin{abstract}
In this paper, we derive a relative volume comparison estimate along Ricci flow and apply it to studying the Gromov-Hausdorff convergence of K\"ahler-Ricci flow on a minimal manifold. This new estimate generalizes Perelman's no local collapsing estimate and can be regarded as an analogue of the Bishop-Gromov volume comparison for Ricci flow.
\end{abstract}

\tableofcontents

\section{Introduction}

In this paper, we establish a relative volume comparison for Ricci flow. This volume comparison is a refinement of Perelman's no local collapsing estimate in \cite{Pe02}. The major advantage is that this volume comparison does not require any non-collapsing conditions on initial metrics, so our new estimate, like the Bishop-Gromov relative volume comparison estimate on a manifold with Ricci curvature bounded from below, can be applied to studying Ricci flow with a collapsing structure. As one special application, we use the relative volume comparison to studying the Gromov-Hausdorff convergence of K\"ahler-Ricci flow on a K\"ahler manifold with semi-ample canonical line bundle and positive Kodaira dimension; see Theorem \ref{KRF: minimal surface} below. In case of Kodaira dimension 1 with toric fibration, we confirm a very important part of the Analytic Minimal Model Program (AMMP); see \cite{Ti08} \cite{SoTi09} for a description of the AMMP. In particular, we solved a conjecture on the convergence of K\"ahler-Ricci flow on minimal elliptic surfaces, which was proposed first in \cite{SoTi07} and reemphasised subsequently in \cite{Ti08} \cite{SoWe13} \cite{To15}.

The Ricci flow was introduced by Hamilton in 1982 \cite{Ha82}. It evolves metrics $g(t)$ on a manifold by
\begin{equation}\label{Ricci flow}
\frac{\partial}{\partial t}g(t)\,=\,-2\Ric(g(t)).
\end{equation}
The no local collapsing theorem of Perelman is crucial for his celebrated works on the Poincar\'e conjecture and geometrization of 3-manifolds \cite{Pe02, Pe03}. It is used to ruling out the possibility that Ricci flow may form finite-time singularity around a surface of positive genus on any closed 3-manifolds, as well as proving many other results where the convergence of Ricci flow is involved. Perelman introduced two new tools, $\mathcal{W}$-entropy and reduced volume, to prove his no local collapsing theorems; see the corresponding Theorem 4.1 and Theorem 8.2 in \cite{Pe02} respectively. Though his second theorem is conceptually weaker, it is more powerful in applications since it holds in a uniform way. Let us recall the theorem.

\begin{theorem}[Perelman \cite{Pe02}]\label{no local collapsing}
For any $A>0$ and dimension $m$ there exists $\kappa=\kappa(m,A)>0$ with the following property. If $g(t)$, $0\le t\le r_0^2$, is a smooth solution to the Ricci flow on an $n$-manifold which has $|Rm|(x,t)\le r_0^{-2}$ for all $(x,t)$ satisfying $d_0(x,x_0)\le r_0$, and the volume of the metric ball $B_{g(0)}(x_0,r_0)$ is at least $A^{-1}r_0^n$, then, for any metric ball $B_{g(r_0^2)}(x,r)\subset B_{g(r_0^2)}(x_0,Ar_0)$ such that $r\le r_0$ and
\begin{equation}\label{parabolic curvature}
|Rm|(y,t)\le r^{-2},\,\mbox{ for all }(y,t)\in B_{g(r_0^2)}(x,r)\times[r_0^2-r^2,r_0^2],
\end{equation}
one has
\begin{equation}\label{volume}
\vol_{g(r_0^2)}\big(B_{g(r_0^2)}(x,r)\big)\,\ge\,\kappa r^m.
\end{equation}
\end{theorem}

Here, we denote by $d_t$ the distance function defined by $g(t)$.

Perelman used the monotonicity of reduced volume along the Ricci flow to prove this theorem. In \cite{ZhQi}, using the loacalized $\mathcal{W}$-entropy and a localized version of Perelman differential Harnack inequality, Q. Zhang proved a uniform local Sobolev inequality for the metric ball $B_{g(r_0^2)}(x_0,Ar_0)$; see Theorem 6.3.2 in \cite{ZhQi}. As a direct consequence, he showed that, in order to estimate the volume of a metric ball at time $t=r_0^2$ as in (\ref{volume}), instead of assuming the curvature condition (\ref{parabolic curvature}) on a parabolic domain as Perelman did, one just needs the scalar curvature estimate at the time slice $t=r_0^2$, namely,
\begin{equation}
R(y,r_0^2)\le r_0^{-2},\,\mbox{ for all }y\in B_{g(r_0^2)}(x,r).
\end{equation}
Recently, Wang \cite{Wa17} gave another proof of this improved no local collapsing theorem and applied it to studying K\"ahler-Ricci flow on smooth minimal models of general type. Let us recall his result as well as some history on K\"ahler-Ricci flow on a smooth minimal model of general type $X$. On such a manifold, one usually considers the normalized K\"ahler-Ricci flow
\begin{equation}
\frac{\partial}{\partial t}\omega(t)=-\Ric(\omega(t))-\omega(t)
\end{equation}
where $\omega(t)$ is the K\"ahler form associated to K\"ahler metric along the flow. It follows from a general criterion of Tian-Zhang \cite{TiZh06} that the flow has a global solution $\omega(t)$ on $X$ for all $t\ge 0$ (also see \cite{Ts}). As part of the Analytic Minimal Model Program (AMMP), it is conjectured in \cite{SoTi09}, \cite{So14} that the K\"ahler-Ricci flow converges in the Gromov-Hausdorff topology to the canonical K\"ahler-Einstein metric on the canonical model $X_{can}$. It follows from the work of Tsuji \cite{Ts}, Tian-Zhang \cite{TiZh06} and Eyssidieux-Guedj-Zeriahi \cite{EGZ09} that there is a unique K\"ahler-Einstein current $\omega_{\KE}$. It is proved by Song \cite{So14} that $\omega_{\KE}$ does define a metric on $X_{can}$. There are many results on the limit current. When $c_1(X)<0$, Cao \cite{Ca85} proved the smooth convergence of the K\"ahler-Ricci flow to the unique K\"ahler-Einstein metric on $X$; for general $X$, Tsuji \cite{Ts} and Tian-Zhang \cite{TiZh06} proved that the K\"ahler-Ricci flow converges to the canonical K\"ahler-Einstein current $\omega_{\KE}$ in the smooth topology outside the non-ample locus of the canonical class $K_X$; when the (complex) dimension is $\le 3$, the authors \cite{TiZh15} proved the Gromov-Hausdorff convergence, so solved the conjecture in low dimensions (also see \cite{GuSoWe15} for an independent proof for minimal surfaces of general type); Guo \cite{Gu15} proved the geometric convergence under additional assumption of lower bounded Ricci curvature. We remark that the work \cite{TiZh15} \cite{Gu15} depends highly on Song's approach to bounding diameter of $\omega_{\KE}$ on the canonical model $X_{can}$ \cite{So14}. Finally, based on the work of Tsuji \cite{Ts} and Tian-Zhang \cite{TiZh06} on local $C^\infty$ convergence and Zhang's scalar curvature bound \cite{ZhZ09} and Song's diameter bound of the limit space, Wang proved in \cite{Wa17} that the diameters of $\omega(t)$ are uniformly bounded under the K\"ahler-Ricci flow in all dimensions. It implies the Gromov-Hausdorff convergence of $(X,\omega(t))$, at least along any subsequences. However, it is not obvious from Wang's work that the Gromov-Hausdorff limit coincides with the canonical model. In general, the problem to identify the Gromov-Hausdorff limit and the canonical model is highly nontrivial. In case of minimal K\"ahler surfaces one can use the fact that the singularities on the canonical model are finite; in higher dimensions, one approach is to make use of the partial $C^0$ estimate (see \cite{TiZh15} for example).

Our work is inspired by \cite{Wa17}. The following is our main result on relative volume comparison which can be regarded as a relative version of Theorem \ref{no local collapsing}.

\begin{theorem}\label{relative noncollapsing}
For any $m$ and $A\ge 1$ there exists $\kappa=\kappa(m,A)>0$ such that the following holds. Let $g(t)$, $0\le t\le r_0^2$, be a solution to the Ricci flow on a compact $m$-manifold $M$ such that
\begin{equation}
|\Ric|\le r_0^{-2},\mbox{ on }B_{g(0)}(x_0,r_0)\times[0,r_0^2].
\end{equation}
Then, for any metric ball $B_{g(r_0^2)}(x,r)\subset B_{g(r_0^2)}(x_0,Ar_0)$ of radius $r\le r_0$ satisfying
\begin{equation}
R(\cdot,r_0^2)\le r^{-2}\,\mbox{ in } B_{g(r_0^2)}(x,r),
\end{equation}
we have the relative volume comparison
\begin{equation}\label{relative volume}
\frac{\vol_{g(r_0^2)}(B_{g(r_0^2)}(x,r))}{r^m}\,\ge\, \kappa\,\frac{\vol_{g(0)}(B_{g(0)}(x_0,r_0))}{r_0^m}.
\end{equation}
\end{theorem}

We will prove this theorem by using a local version of Perelman's entropy functional. The application of localized entropy to Ricci flow appears firstly in \cite{ZhQi} and also in \cite{Wa17}.

\begin{remark}
Because the Ricci curvature is uniformly bounded on the space-time domain $B_{g(0)}(x_0,r_0)\times[0,r_0^2]$, the size of the metric balls $B_{g(0)}(x_0,e^{-1}r_0)$ and $B_{g(r_0^2)}(x_0,e^{-1}r_0)$ and their volumes are comparable under Ricci flow. So, the relative volume comparison (\ref{relative volume}) is equivalent to
\begin{equation}
\frac{\vol_{g(r_0^2)}(B_{g(r_0^2)}(x,r))}{r^n}\ge \kappa(m,A)\,\frac{\vol_{g(r_0^2)}(B_{g(r_0^2)}(x_0,e^{-1}r_0))}{r_0^n}.
\end{equation}
It is a relative volume comparison of metric balls in $(M,g(r_0^2))$.
\end{remark}

\begin{remark}
The constant $\kappa(m,A)$ can be calculated. See Remark \ref{constant 2} below.
\end{remark}

One application of Theorem \ref{relative noncollapsing} is to study the Gromov-Hausdorff convergence of the K\"ahler-Ricci flow on a smooth minimal model of Kodaira dimension between $1$ and $n-1$, where $n$ is the dimension of the manifold. In this case, the K\"ahler-Ricci flow admits a collapsing structure when the time goes to infinity, hence, previous no local collapsing results do not apply.

From now on, we let $X$ be a K\"ahler manifold of dimension $n$. Suppose that the Kodaira dimension is positive and strictly less than $n$. Inspired by the Abundance conjecture in algebraic geometry, we assume that $K_X$ is semi-ample. Let $\pi:X\rightarrow X_{can}$ be a holomorphic fibration onto its canonical model defined via a basis of $H^0(X,\ell K_X)$ for some $\ell>>1$. Recall that in \cite{SoTi07, SoTi12}, Song and the first named author constructed the generalized K\"ahler-Einstein current $\omega_{\GKE}$ on the canonical model $X_{can}$. Let $S\subset X_{can}$ denote the set of singular values of $\pi$. Then $\omega_{\GKE}$ is smooth on $X_{can}\backslash S$ and satisfies
\begin{equation}
\Ric(\omega_{\GKE})\,=\,-\omega_{\GKE}+\omega_{\WP},\,\mbox{ on }X_{can}\backslash S,
\end{equation}
where $\omega_{\WP}$ is the Weil-Petersson form on $X_{can}\backslash S$. Song and the first named author also proved the convergence of the K\"ahler-Ricci flow to $\pi^*\omega_{\GKE}$ in the current sense \cite{SoTi07, SoTi12}, furthermore, they proved the $C^0$-convergence on the potential level and in the case when $X$ is an elliptic surface the $C_{loc}^{1,\alpha}$-convergence of potentials on $X_{\reg}=\pi^{-1}(X_{can}\backslash S)$ for any $\alpha < 1$. In \cite{FoZh15}, Fong-Zhang proved the $C^{1,\alpha}$-convergence of potentials when $X$ is a global submersion over $X_{can}$ and showed the Gromov-Hausdorff convergence in the special case. In \cite{ToWeYa17} Tosatti-Weinkove-Yang improved the estimate and showed that the metric $\omega(t)$ converges to $\pi^*\omega_{\GKE}$ in the $C^0_{loc}$-topology on $X_{\reg}$. Moreover, Tosatti-Weinkove-Yang \cite{ToWeYa17} also proved that the restricted metric $\omega(t)|_{X_s}$ converges (up to scalings) in the $C^0$-topology to the unique Ricci flat metric on the fibre $X_s$ for any regular value $s$; this result is improved to be smooth convergence by Tosatti-Zhang in \cite{ToZh16}. In general, it is conjectured that the metric $\omega(t)$ should converge smoothly to $\pi^*\omega_{\GKE}$ on $X_{\reg}$; see \cite{SoTi07} for the case of K\"ahler surfaces. When the generic fibres of $\pi$ are tori or more generally finite quotients of tori, the conjecture is known to be true thanks to the work of Fong-Zhang \cite{FoZh15}, Hein-Tosatti \cite{HeTo15} and Tosatti-Zhang \cite{ToZh16} by developing parabolic version of certain arguments in \cite{GrToZh13}; also see Tosatti's note \cite{To15} for clearer and more unified discussions. In \cite{Gi14}, Gill considered the special case when $X$ is a product of a manifold of negative $c_1$ and a flat manifold. Also see one example of product elliptic surface in the note by Song-Weinkove \cite{SoWe13}.


Although many analytic aspects of the convergence of the K\"ahler-Ricci flow on $X$ have been known, we almost have no knowledge about the geometric convergence when the singular set $S$ is nonempty, even in the simplest case when $X$ is a minimal surface of Kodaira dimension 1. The difficulty is how to control the size of the singular fibres effectively under the K\"ahler-Ricci flow. Our volume comparison theorem leads to an approach to solve the difficulty. As one example we can prove the following theorem.

\begin{theorem}\label{KRF: minimal model}
Let $X$ be a K\"ahler manifold with $K_X$ semi-ample and Kodaira dimension 1. Suppose a K\"ahler-Ricci flow $\omega(t)$ on $X$ satisfies
\begin{equation}\label{Ricci: minimal model}
|\Ric|\le \Lambda,\,\mbox{ on }\pi^{-1}(U)\times[0,\infty),
\end{equation}
uniformly for some $\Lambda<\infty$ on a domain $U\subset X_{can}\backslash S$. Then $(X,\omega(t))$ converges in the Gromov-Hausdorff topology to the generalized K\"ahler-Einstein metric space $(X_{can},d_{\GKE})$.
\end{theorem}

In the proof we also make crucial use of the local estimate of the metric $\omega_{\GKE}$ around the singular points proved by Y.S. Zhang \cite{Zh17}; see also \cite{ZhZh16} for the case when $X$ is a minimal surface. It follows in particular that the generalized K\"ahler-Einstein metric space $(X_{can},d_{\GKE})$ is nothing but the metric completion of $(X_{can}\backslash S,\omega_{\GKE})$; see Section 6 for further discussion of this space. By the work of \cite{FoZh15, GrToZh13, HeTo15, ToZh16}, we know that if the fibration $\pi:X\rightarrow X_{can}$ has generic fibres finite quotient of complex tori, then, for any compact subset $K\subset X_{can}\backslash S$, there exists $C=C(\omega(0),K)$ such that
\begin{equation}\label{e: 615}
|\Ric(t)|\le C,\,\mbox{ on }\pi^{-1}(K)\times[0,\infty),
\end{equation}
and that $\omega(t)\rightarrow\pi^*\omega_{\GKE}$ smoothly on $K$ as $t\rightarrow\infty$. So we immediately have

\begin{corollary}
Let $X$ be a K\"ahler manifold with $K_X$ semi-ample and Kodaira dimension 1. If the generic fibres of the holomorphic fibration $\pi:X\rightarrow X_{can}$ are finite quotients of tori, then any K\"ahler-Ricci flow on $X$ converges in the Gromov-Hausdorff topology to the generalized K\"ahler-Einstein metric space $(X_{can},d_{\GKE})$.
\end{corollary}

In particular, we solve a conjecture in \cite{SoTi07} for elliptic surfaces.

\begin{corollary}\label{KRF: minimal surface}
Let $X$ be a smooth minimal elliptic surface of Kodaira dimension 1 and $X_{can}$ be its canonical model. Any K\"ahler-Ricci flow $\omega(t)$ on $X$ converges in the Gromov-Hausdorff topology to the generalized K\"ahler-Einstein metric space $(X_{can},d_{\GKE})$.
\end{corollary}

The case of higher Kodaira dimension is not totally clear. On one hand, one can follow Song-Tian's construction to define the generalized K\"ahler-Einstein metric space and furthermore, consider the geometric convergence of the K\"ahler-Ricci flow to this metric space. On the other hand, according to the AMMP proposed by Song and the first named author \cite{SoTi12, SoTi09}, the generalized K\"ahler-Einstein metric space should be nothing but the canonical model of the manifold. In case of general type one can probably use the partial $C^0$ estimate to identify the K\"ahler-Einstein metric space. However, in the case when the Kodaira dimension is strictly less than the dimension of the manifold, there has been no effective method to identify the metric space. See Section 6 for some further discussions.

We end the introduction with a brief discussion on the organization of the paper. In Sections 2-4 we prove the main Theorem \ref{relative noncollapsing}. The key ingredient is the Li-Yau type lower bound of the conjugate heat kernel along Ricci flow. The proof uses the comparison principle of Cheeger-Yau \cite{ChYa81} on the heat kernel and the Harnack inequalities along Ricci flow: Perelman's Harnack \cite{Pe02} for heat equations and Kuang-Zhang's Harnack \cite{KuZhQ} for conjugate heat equations. In Section 5, we give a proof of Theorem \ref{KRF: minimal model} as an application of Theorem \ref{relative noncollapsing}. In Section 6, we discuss some open problems on K\"ahler-Ricci flow on minimal models of higher dimensions.

\noindent
{\bf Acknowledgement}: The second named author would like to thank Y.S. Zhang for his discussion on the generalized K\"ahler-Einstein metric, and Q. S. Zhang and S.J. Zhang for pointing out some typos and an inaccuracy in the proof of Lemma 3.1.

\section{Local entropy and Ricci flow}

In \cite{Pe02} Perelman introduced the entropy functional on a Riemannian manifold and  proved its monotonicity under Ricci flow, then he applied it to prove the no local collapsing theorem 4.1 \cite{Pe02}. In \cite{ZhQi} and \cite{Wa17} the authors considered the localization of the entropy and improved Perelman's no local collapsing theorem. Our aim is to prove a relative version of Perelman's no local collapsing theorem. In this section we recall the basic notations of the localized entropy and prove some related estimates.

Let $(M,g)$ be a compact Riemannian $m$-manifold. Recall Perelman's $\mathcal{W}$-functional \cite{Pe02}
\begin{equation}
\mathcal{W}(g,f,\tau)=\int_M\big[\tau(R+|\nabla f|^2)+f-m\big](4\pi\tau)^{-m/2}e^{-f}dv.
\end{equation}
After putting $u=(4\pi\tau)^{-m/4}e^{-f/2}$, it can be rewritten as
\begin{equation}
\mathcal{W}(g,u,\tau)=\tau\int_M(Ru^2+4|\nabla u|^2)dv-\int_M u^2\log u^2dv-\frac{m}{2}\log(4\pi\tau)-m.
\end{equation}
Let $\Omega$ be any bounded domain of $M$. We define the local entropy (cf. \cite{ZhQi} and \cite{Wa17})
\begin{equation}
\mu_\Omega(g,\tau)=\inf\big\{\mathcal{W}(g,u,\tau)\big|u\in C_0^\infty(\Omega),\,\int_\Omega u^2=1\big\}.
\end{equation}
When $\Omega=M$ it is exactly the entropy $\mu(g,\tau)$ introduced in Perelman's paper \cite{Pe02}. It can be checked easily that the entropy satisfies the scaling invariance
\begin{equation}\label{invariance}
\mu_{\Omega}(cg,c\tau)=\mu_{\Omega}(g,\tau)
\end{equation}
for any positive constant $c$. Moreover, the entropy satisfies the monotonicity
\begin{equation}
\mu_{\Omega'}(g,\tau)\ge\mu_{\Omega}(g,\tau)
\end{equation}
for any subdomain $\Omega'\subset\Omega$. When $\Omega$ has smooth boundary, there always exists a minimizer of $\mu_\Omega$ which is smooth in $\Omega$ and continuous up to the boundary \cite{Ro81}.

Following Perelman \cite{Pe02} we also define the local energy functional, for any $a>0$,
\begin{equation}
\lambda_{a,\Omega}(g)=\inf\big\{\int_\Omega(Ru^2+a|\nabla u|^2)dv\big| u\in C_0^\infty(\Omega),\,\int_\Omega u^2dv=1\big\}.
\end{equation}
It is obviously that $\lambda_{a,\Omega}$ is the smallest eigenvalue of the operator $R-a\triangle$ with Dirichlet condition. It satisfies the scaling property $\lambda_{a,\Omega}(cg)=c^{-1}\lambda_{a,\Omega}(g)$.

\subsection{Basic estimates}

The entropy $\mu_\Omega$ is roughly equivalent to the Log-Sobolev inequality of the domain $\Omega$. We will show that it is also crucially related to the volume ratio of the domain. For some further estimates of $\mu_\Omega$ we refer to \cite{ZhQi} and \cite{Wa17}.

\begin{lemma}
For any domain $\Omega$, metric $g$ and $\tau>0$ we have
\begin{equation}\label{e: 215}
\mu_\Omega(g,\tau)\le\tau\lambda_{4,\Omega}+\log\vol_g(\Omega)-\frac{m}{2}\log\tau+C(m)
\end{equation}
and
\begin{equation}\label{e: 216}
\mu_\Omega(g,\tau)\ge\tau\lambda_{3,\Omega}+\log\vol_g(\Omega)-m\log C_s(\Omega)-C(m),
\end{equation}
where $C_s(\Omega)$ is the Sobolev constant of $\Omega$ in the sense that
\begin{equation}
\bigg(\fint_\Omega f^{\frac{2m}{m-2}}dv\bigg)^{\frac{m-2}{2m}}\le C_s\bigg(\fint_\Omega|\nabla f|^2dv\bigg)^{1/2},\,\forall f\in C_0^\infty(\Omega).
\end{equation}
\end{lemma}
\begin{proof}
We adopt the arguments from \cite{Zh07}. We first prove the upper bound. Let $\lambda_a=\lambda_{a,\Omega}$ for simplicity. Let $u$ be the eigenfunction of $\lambda_4$. Then, by definition of $\mu_\Omega$ and the trivial fact $-x\log x\le 1$ for any $x>0$, we have
\begin{eqnarray}\nonumber
\mu_\Omega(g,\tau)\le\tau\lambda_4(g)+\vol_g(\Omega)-\frac{m}{2}\log\tau+C(m),\,\forall\tau>0.
\end{eqnarray}
Then we apply the scaling invariance of $\lambda$ and $\mu_\Omega$ to get
\begin{eqnarray}\nonumber
\mu_\Omega(cg,c\tau)&\le& c\tau\lambda_4(cg)+\vol_{cg}(\Omega)-\frac{m}{2}\log(c\tau)+C(m)\nonumber\\
&=&\tau\lambda_4(g)+c^{m/2}\cdot\vol_{g}(\Omega)-\frac{m}{2}\log\tau-\frac{m}{2}\log c+C(m)\nonumber.
\end{eqnarray}
In particular if we choose $c=\vol_g(\Omega)^{-2/m}$, then we get
$$\mu(g,\tau)=\mu(cg,c\tau)\le\tau\lambda_4(g)+\log\vol_{g}(\Omega)-\frac{m}{2}\log\tau+C(m).$$

Then we prove the lower bound. We shall apply the Sobolev inequality and Jensen inequality: for any $u\in C_0^\infty(\Omega)$ with $\int_\Omega u^2dv=1$, we have, w.r.t. the measure $d\mu=u^2dv$,
\begin{eqnarray}
-\int_\Omega u^2\log u^2dv&=&-\int_\Omega\log u^2d\mu=-\frac{m-2}{2}\int_\Omega\log u^{\frac{2m}{m-2}-2}d\mu\nonumber\\
&\ge&-\frac{m-2}{2}\log\int_\Omega u^{\frac{2m}{m-2}-2}d\mu\nonumber\\
&=&-m\log\bigg(\int_\Omega u^{\frac{2m}{m-2}}dv\bigg)^{\frac{m-2}{2m}}\nonumber\\
&=&-m\log\bigg(\fint_\Omega u^{\frac{2m}{m-2}}dv\bigg)^{\frac{m-2}{2m}}-\frac{m-2}{2}\log\vol(\Omega)\nonumber\\
&\ge&-m\log C_s-\frac{m}{2}\log\fint_\Omega|\nabla u|^2dv-\frac{m-2}{2}\log\vol(\Omega)\nonumber\\
&=&-m\log C_s-\frac{m}{2}\log\int_\Omega|\nabla u|^2dv+\log\vol(\Omega)\nonumber.
\end{eqnarray}
Thus,
\begin{eqnarray}
\mathcal{W}(g,u,\tau)&\ge&\tau\int_\Omega(Ru^2+4|\nabla u|^2)dv-\frac{m}{2}\log\int_\Omega|\nabla u|^2dv\nonumber\\
&&\hspace{2cm}-m\log C_s+\log\vol(\Omega)-\frac{m}{2}\log\tau-C(m)\nonumber\\
&\ge&\tau\lambda_{3}+\tau\int_\Omega|\nabla u|^2dv-\frac{m}{2}\log\bigg(\tau\int_\Omega|\nabla u|^2dv\bigg)\nonumber\\
&&\hspace{2cm}+\log\vol(\Omega)-m\log C_s-C(m).\nonumber
\end{eqnarray}
Finally we use the easy fact
$$x-\frac{m}{2}\log x\ge\frac{1}{2}+\frac{m}{2}\log 2,\,\forall x>0.$$
to conclude that
$$\mathcal{W}(g,u,\tau)\ge\tau\lambda_{3}+\log\vol(\Omega)-m\log C_s-C(m).$$
It gives the lower bound of $\mu_\Omega(g,\tau)$.
\end{proof}

\begin{corollary}\label{entropy-volume: 1}
Let $B(x,2r)\subset M$ be a metric ball with $\partial B(x,2r)\neq\emptyset$. If
\begin{equation}\label{e: 220}
\Ric\ge-r^{-2},\mbox{ in }B(x,2r),
\end{equation}
then
\begin{equation}\label{e: 221}
\log\frac{\vol_g(B(x,r))}{r^m}\le\inf_{0<\tau\le r^2}\mu_{B(x,r)}(g,\tau)+C(m).
\end{equation}
\end{corollary}
\begin{proof}
It is well-known that under the Ricci lower bound (\ref{e: 220}) the Sobolev constant satisfies the uniform bound
\begin{equation}
C_s(B(x,r))\le C(m)\cdot r.
\end{equation}
Moreover, the eigenvalue $\lambda_{3,B(x,r)}$ admits the trivial lower bound
\begin{equation}
\lambda_{3,B(x,r)}\ge\inf_{B(x,r)}R\ge-mr^{-2}.
\end{equation}
Substituting both estimates into the formula (\ref{e: 216}) gives the desired result.
\end{proof}

\begin{corollary}\label{entropy-volume: 2}
Let $B(x,r)$ be a metric ball with $\partial B(x,r)\neq\emptyset$. If the scalar curvature
\begin{equation}\label{e: 217}
R\le mr^{-2},\mbox{ in }B(x,r),
\end{equation}
then,
\begin{equation}\label{e: 218}
\log\frac{\vol(B(x,r))}{r^m}\ge \inf_{0<\tau\le r^2}\mu_{B(x,r)}(g,\tau)-C(m).
\end{equation}
\end{corollary}
\begin{proof}
The proof is essentially contained in Remark 13.13 of \cite{KlLo08}. We sketch its proof here for the sake of completeness. Assume in a prior that
\begin{equation}\label{doubling}
\vol(B(x,r))\le 4^m\vol(B(x,\frac{r}{2})).
\end{equation}
Then we claim that $\lambda_{4,B}\le C(m)\cdot r^{-2}$ under the assumption (\ref{e: 217}). Actually, it can be checked by choosing the test function $u$, in the definition of $\lambda_{4,B}$, which is a positive constant in $B(x,\frac{r}{2})$ and decreases linearly to 0 on $B(x,r)\backslash B(x,\frac{r}{2})$ such that $\int_Bu^2=1$. Then, by (\ref{e: 215}) we get, under (\ref{doubling}),
$$\mu_{B(x,r)}(g,r^2)\le\log\frac{\vol(B(x,r))}{r^m}+C(m).$$

In general, there always exists a minimum integer $k_0\ge 0$ such that (\ref{doubling}) holds for the radius $2^{-k_0}r$. Then, for this $k_0$ we have
\begin{equation}
\mu_{B(x,2^{-k_0}r)}(g,2^{-2k_0}r^2)\le\log\frac{\vol(B(x,2^{-k_0}r))}{2^{-k_0m}r^m}+C(m).
\end{equation}
By the monotonicity of $\mu$ with respect to the subdomains,
\begin{equation}
\mu_{B(x,r)}(g,2^{-2k_0}r^2)\le\mu_{B(x,2^{-k_0}r)}(g,2^{-2k_0}r^2).
\end{equation}
On the other hand, by the definition of $k_0$, we have by induction that
\begin{equation}
\frac{\vol(B(x,r))}{r^m}\ge\frac{\vol(B(x,2^{-k_0}r))}{2^{-k_0m}r^m}.
\end{equation}
The required volume ratio estimate follows from these three formulas.
\end{proof}

\subsection{Partial monotonicity under Ricci flow}

Following the calculations in Section 6.3 of \cite{ZhQi} and Section 5 of \cite{Wa17}, we sketch a partial monotonicity of the local entropy under Ricci flow.

Let $(M,g(t)),0\le t\le T,$ be a solution to the Ricci flow on a compact $m$-manifold. Let $r$ be a radius such that $r^2\le T$. Let $A\ge 1$ be a constant and $\Omega=B_{g(T)}(x_0,Ar)$ be a ball at time $T$ such that $\partial \Omega\neq\emptyset$. By approximating by smooth domains bigger than $\Omega$ and monotonicity of $\mu$ with respect to domains, in the following calculation we may assume that $\partial\Omega$ is smooth. Define $\tau(t)=\tau_0+T-t$ for some $\tau_0>0$.

By Rothaus \cite{Ro81}, there exists a minimizer of $\mu_\Omega(g(T),\tau_0)$, say $u\in C^\infty(\Omega)\cap C^0(\overline{\Omega})$ which vanishes identically on $\partial\Omega$. So $u$ can be viewed as a function on $M$ which vanishes outside $\Omega$. The Euler-Lagrange equation reads
\begin{equation}
\tau_0\big(-4\triangle u+Ru\big)-u\log u^2=\big(\mu_\Omega(g(T),\tau_0)+\frac{m}{2}\log(4\pi\tau)+m\big)\cdot u.
\end{equation}
Now let $v(t)$ be the solution to the backward heat equation
\begin{equation}
\frac{\partial}{\partial t}v=-\triangle v+Rv,
\end{equation}
with initial value $v(T)=u^2(T)$. By maximal principle we have $v>0$ for all $0\le t<T$. Put $u(t)=\sqrt{v(t)}$ when $t<T$. Then we have
\begin{equation}
\int_Mu(t)^2dv_{g(t)}=1,
\end{equation}
and the Li-Yau-Perelman Harnack inequality, cf. \cite[(6.3.30)]{ZhQi} or \cite[Theorem 4.2]{Wa17},
\begin{equation}\label{L-Y-H-P harnack}
\tau(Ru-4\triangle u)-u\log u^2-\big(\mu_\Omega(g(T),\tau_0)+\frac{m}{2}\log(4\pi\tau)+m\big)\cdot u\le 0,
\end{equation}
at any time $0\le t<T$. Now, at any time $t\in[T-r^2,T]$ we let $\eta\in C_0^\infty(B_{g(t)}(x_0,r))$ be a cut-off function such that
\begin{equation}
0\le\eta\le 1,\,\eta\equiv 1\mbox{ on }B_{g(t)}(x_0,2^{-1}r),
\end{equation}
and $\|\nabla\eta\|_{g(t)}\le 4r^{-1}$. We put $\tilde{u}=\delta^{-1}\eta u(t)$ where $\delta=\|\eta u(t)\|_{L^2(g(t))}$ such that
$$\tilde{u}\in C_0^\infty\big(B_{g(t)}(x_0,r)\big)\,\mbox{ and }\int_M\tilde{u}^2dv_{g(t)}=1.$$
Obviously we have $\delta^2\ge\int_{B_{g(t)}(x_0,\frac{r}{2})}u^2(t)dv_{g(t)}$. Then,
by definition,
\begin{equation}
\mathcal{W}\big(g(t),\tilde{u},\tau\big)=\tau\int_M(R\tilde{u}^2+4|\nabla \tilde{u}|^2)dv-\int_M \tilde{u}^2\log \tilde{u}^2dv-\frac{m}{2}\log(4\pi\tau)-m\nonumber.
\end{equation}
By a straightforward calculation,
\begin{eqnarray}
\mathcal{W}\big(g(t),\tilde{u},\tau\big)
&=&\int_M \delta^{-2}\eta^2\big[\tau(Ru^2-4u\triangle u)-u^2\log u^2\big]+4\tau \delta^{-2}\int_M u^2|\nabla \eta|^2\nonumber\\
&&-\int_M\delta^{-2}\eta^2\log(\delta^{-2}\eta^2)\cdot u^2-\frac{m}{2}\log(4\pi\tau)-m\nonumber.
\end{eqnarray}
Then, by the Li-Yau-Perelman Harnack inequality (\ref{L-Y-H-P harnack}),
\begin{equation}
\mathcal{W}\big(g(t),\tilde{u},\tau\big)\le\mu_\Omega(g(T),\tau_0)+4\tau \delta^{-2}\int_M u^2|\nabla \eta|^2-\int_M\delta^{-2}\eta^2\log(\delta^{-2}\eta^2)\cdot u^2+1\nonumber.
\end{equation}
Using the trivial fact $-x\log x\le 1$ for any $x>0$, we obtain the following lemma which is essentially Theorems 5.1 in \cite{Wa17}.

\begin{lemma}
Under above assumption we have that
\begin{equation}\label{relative volume ratio: 2}
\mu_{B}\big(g(t),\tau(t)\big)\le\mu_\Omega(g(T),\tau_0)+200\cdot\bigg(\int_{B_{g(t)(x_0,\frac{r}{2})}}u^2(t)dv_{g(t)}\bigg)^{-1}
\end{equation}
whenever $T-r^2\le t\le T$ and $\tau_0\le r^2$, where $B=B_{g(t)}(x_0,r)$ and $\Omega=B_{g(T)}(x_0,Ar)$.
\end{lemma}

The remaining question is how to estimate the integration $\int_{B_{g(t)}(x_0,\frac{r}{2})}u^2(t)dv_{g(t)}$ from below. This is the full motivation of the following section on the heat kernel estimate.

\section{Heat kernel lower bound to the conjugate heat equation}

The heat kernel estimate is one important topic in the study of Ricci flow and many remarkable applications have been found, cf. \cite{Ba16, BaRiWi17, BaZh15-1, BaZh15-2, CaZh, HeNa13, KuZhQ, Pe02, ZhQ07, ZhQ10, ZhQ} etc. In all these works the non-collapsing assumption is essential. In this section we establish a partial Li-Yau type heat kernel estimate where the volumes of metric balls involve. It is the crucial technique of this paper and can be applied to the Ricci flow with a collapsing structure.

In this section we assume the space-time scale is 1. Let $g(t)$, $0\le t\le 1$, be a Ricci flow on a compact $m$-manifold $M$. Let $B_0=B_{g(0)}(x_0,1)$ be a metric ball at the initial time $t=0$, centered at a point $x_0$, such that $\partial B\neq\emptyset$. Assume that
\begin{equation}\label{Ricci: normalized}
|\Ric|\le 1,\,\mbox{ on }B_0\times[0,1].
\end{equation}
Then, by the Ricci flow equation we have
\begin{equation}\label{e: 301}
e^{-1}\cdot g(t)\le g(1)\le e\cdot g(t),\,\mbox{ on } B_0\times[0,1].
\end{equation}
So,
\begin{equation}
B_{g(t)}(x_0,e^{-2})\subset B_{g(s)}(x_0,e^{-1})\subset B_0,\,\forall t,s\in[0,1].
\end{equation}

Let $H(x,t';y,t)$, $0\le t'<t\le 1$, be the heat kernel to the conjugate heat equation which satisfies
\begin{equation}
-\frac{\partial}{\partial t'}H=\triangle_{g(t'),x}H-R(x,t')\cdot H
\end{equation}
with initial value
\begin{equation}
\lim_{t'\nearrow t}H(x,t';y,t)=\delta_{g(t),y}(x),
\end{equation}
and
\begin{equation}
\frac{\partial}{\partial t}H=\triangle_{g(t),y}H,
\end{equation}
with initial value
\begin{equation}
\lim_{t\searrow t'}H(x,t';y,t)=\delta_{g(t'),x}(y)
\end{equation}
where $\delta_{g(t),x}$ is the Dirac function concentrated at $(x,t)$ with respect to the Riemannian measure of $g(t)$. For the existence of the heat kernel and its estimates under Ricci flow we refer to \cite{Chetal} \cite{ZhQi} and the references therein.

In the following we will use $d_t$ to denote the distance function of $g(t)$.

\subsection{Heat kernel integral bound}

In this subsection we adopt Li-Yau argument \cite{LiYa86}, following the idea of Cheeger-Yau \cite{ChYa81}, to derive an integral lower bound of $H$. More references for Li-Yau estimate are \cite{Li} \cite{SchYa}.

Let $\varphi:\mathbb{R}\rightarrow[0,\infty)$ be a smooth function satisfying
\begin{equation}
\varphi(r)=1,\,\forall r\le e^{-5};\;\varphi(r)=0,\,\forall r\ge e^{-4};\,\varphi'\le 0.
\end{equation}
The function
\begin{equation}
u(y,t)=\int_MH(x,0;y,t)\cdot\varphi(d_0(x_0,x))~dv_{g(0)}(x)
\end{equation}
satisfies the forward heat equation
\begin{equation}
\frac{\partial}{\partial t}u=\triangle_{g(t)} u
\end{equation}
with initial value
$$\lim_{t\rightarrow 0}u(y,t)=\varphi(d_0(x_0,y)).$$
We define the comparison function in the Euclidean space $\mathbb{R}^m$ as follows
\begin{equation}
\bar{u}(\zeta,t)=:\int_{\mathbb{R}^m}(4\pi t)^{-m/2}e^{-\frac{|\xi|^2}{4t}}\varphi(|\xi-\zeta|)d\xi.
\end{equation}
The function $\bar{u}$ is trivially a radial function in the space factor. So $\bar{u}$ determines a function $\bar{u}(r,t)$ for any $r\ge 0$ and $t\ge 0$, by simply setting $\bar{u}(r,t)=\bar{u}(|\zeta|,t)$ whenever $r=|\zeta|$. In the following we use $\bar{u}_t$ and $\bar{u}_r$ to denote the derivatives in the variables $t$ and $r$ respectively. It is easy to check that
\begin{equation}\label{e: 311}
\bar{u}_{r}(r,t)\le 0,
\end{equation}
for any $r,t\ge 0$. Moreover, $\bar{u}$ satisfies the heat equation on $\mathbb{R}^m$
\begin{equation}\nonumber
\bar{u}_t=\bar{u}_{rr}+\frac{n-1}{r}\cdot \bar{u}_r.
\end{equation}
It follows that
\begin{equation}\label{e: 312}
\bar{u}_t\le\bar{u}_{rr}.
\end{equation}
Now we define the comparison function on $M$,
\begin{equation}
\bar{u}(y,t)=\bar{u}(d_t(x_0,y),t),\,\forall y\in M.
\end{equation}
It satisfies the initial condition
\begin{equation}\nonumber
\bar{u}(y,0)=\varphi(d_0(x_0,y)),\,\forall y\in M.
\end{equation}

\begin{lemma}
Assume {\rm(\ref{Ricci: normalized})}. Then the function $\bar{u}$ satisfies
\begin{equation}\label{e: 315}
\frac{\partial}{\partial t}\bar{u}\le\triangle_{g(t)}\bar{u}+C(m)
\end{equation}
in the barrier sense.
\end{lemma}
\begin{proof}
The Laplacian at time $t$ is given by
\begin{eqnarray}\label{e: 313}
\triangle\bar{u}=\bar{u}_r\cdot\triangle d_t+\bar{u}_{rr};
\end{eqnarray}
the derivation in $t$ is given by
\begin{equation}\label{e: 314}
\frac{\partial}{\partial t}\bar{u}=\bar{u}_t+\bar{u}_r\cdot\frac{\partial d_t}{\partial t}.
\end{equation}

We discuss two independent cases according to $r\ge e^{-1}$ or not. When $r\ge e^{-1}$, one can apply Lemma 8.3 (a) in \cite{Pe02}, together with the assumption (\ref{Ricci: normalized}), to derive
$$\frac{\partial d_t}{\partial t}\ge\triangle d_t-2.$$
Thus, together with (\ref{e: 311}) and (\ref{e: 312}), the formulas (\ref{e: 313}) and (\ref{e: 314}) yield
$$\frac{\partial}{\partial t}\bar{u}\le\triangle\bar{u}+\bar{u}_t-\bar{u}_{rr}-2\bar{u}_r\le\triangle\bar{u}
-2\bar{u}_r.$$
Now the required estimate (\ref{e: 315}) follows from the estimate
$$-\bar{u}_r(r,t)\le\int_{\mathbb{R}^m}(4\pi t)^{-m/2}e^{-\frac{|\xi|^2}{4t}}|\varphi'|(|\zeta-\xi|)d\xi\le C$$
where $C$ is a universal constant, $\zeta$ is any point with $|\zeta|=r$.

When $r< e^{-1}$, one can use the local Taylor expansion to get
$$-\bar{u}_r(r,t)\le C(m)\cdot r.$$
On the other hand, since $\Ric\ge-1$ on $B_{g(t)}(x_0,e^{-1})$, the Laplacian comparison gives $\triangle d_t\le\frac{C(m)}{r}$ in the barrier sense. Substituting the estimates into (\ref{e: 313}) to get
$$\triangle\bar{u}\ge \bar{u}_{rr}-C(m)$$
in the barrier sense. Then notice that the term $\frac{\partial d_t}{\partial t}$ in (\ref{e: 314}) admits the estimate
$$\frac{\partial d_t}{\partial t}(x,t)\ge\inf\bigg\{-\int_\gamma\Ric(\dot{\gamma},\dot{\gamma})\bigg|\gamma\mbox{ is minimal geodesic connecting $x_0$ and $x$}\bigg\}\ge-r.$$
Thus,
\begin{equation}\nonumber
\frac{\partial}{\partial t}\bar{u}\le \bar{u}_t+C(m)\le\bar{u}_{rr}+C(m)\le\triangle\bar{u}+C(m)
\end{equation}
in the barrier sense.
\end{proof}

The function $\bar{\bar{u}}(y,t)=\bar{u}(y,t)-Ct$ satisfies
\begin{equation}\nonumber
\frac{\partial}{\partial t}\bar{\bar{u}}\le\triangle_{g(t)}\bar{\bar{u}}
\end{equation}
and the initial condition $\bar{\bar{u}}(y,0)=\varphi(d_0(x_0,y))$. By maximum principle we have
\begin{equation}
u(y,t)\ge\bar{\bar{u}}(y,t),\,\forall (y,t)\in M\times[0,1].
\end{equation}
By approximation we may choose $\varphi$ as the characteristic function on $(-\infty,e^{-4}]$. In particular we have the following consequence.

\begin{corollary}
Under the assumption {\rm(\ref{Ricci: normalized})} we have
\begin{equation}
\int_{B_{g(0)}(x_0,e^{-4})}H(x,0;y,t)~dv_{g(0)}(x)\ge\int_{B(\zeta,e^{-4})}(4\pi t)^{-m/2}e^{-\frac{|\xi|^2}{4t}}d\xi-C(m)\cdot t,
\end{equation}
where $\zeta\in\mathbb{R}^m$ is a point with $|\zeta|=d_t(x_0,y)$ and $B(\zeta,e^{-4})$ is the metric ball of radius $e^{-4}$ centered at $\zeta$ in $\mathbb{R}^m$.
\end{corollary}

In particular, when $d_t(x_0,y)\le e^{-5}$ the integration part on the right hand side tends to 1 as $t\rightarrow 0$. So we have

\begin{corollary}
Under the assumption {\rm(\ref{Ricci: normalized})} we have
\begin{equation}\label{e: 316}
\int_{B_{g(0)}(x_0,e^{-4})}H(x,0;y,t)~dv_{g(0)}(x)\ge\frac{1}{2},
\end{equation}
whenever $d_t(x_0,y)\le e^{-5}$ and $0\le t\le t_0$, where $t_0=t_0(m)\le\frac{1}{200}$ is a positive constant.
\end{corollary}

\begin{remark}
For static metric case, the local curvature condition (\ref{Ricci: normalized}) is not sufficient to derive the heat kernel lower bound (\ref{e: 316}).
\end{remark}

A same argument by replacing the time 0 by $t'<1$ gives the following

\begin{corollary}\label{Cheeger-Yau}
Under the assumption {\rm(\ref{Ricci: normalized})} we have
\begin{equation}\label{e: 317}
\int_{B_{g(t')}(x_0,e^{-4})}H(x,t';y,t)~dv_{g(t')}(x)\ge\frac{1}{2},
\end{equation}
whenever $d_{t}(x_0,y)\le e^{-5}$ and $0\le t'< t\le t'+t_0$, where $t_0=t_0(m)\le\frac{1}{200}$ is a positive constant depending only on the dimension $m$.
\end{corollary}

\subsection{Li-Yau heat kernel lower bound}

First of all we recall the Harnack inequality to the conjugate heat kernel which was proved by Kuang-Zhang \cite{KuZhQ}; see also Corollary 6.4.1 in \cite{ZhQi}. We fix the time $t=1-t_0$ and any point $y\in M$. Then we have,
\begin{equation}\label{e: 321}
H(x_2,t_2';y,1-t_0)\le H(x_1,t_1';y;1-t_0)\cdot\bigg(\frac{1-t_0-t_1'}{1-t_0-t_2'}\bigg)^{3m/2}\cdot\exp\bigg(\frac{L(x_1,t_1';x_2,t_2')}{2(t_2'-t_1')}\bigg)
\end{equation}
for any $x_1,x_2$ and $\frac{1}{2}\le t_1'<t_2'< 1-t_0$. Here,
\begin{equation}\label{e: 322}
L(x_1,t_1';x_2,t_2')=\inf_\gamma\int_{t_1'}^{t_2'}\bigg((t_2'-t_1')^2\cdot R(\gamma(t'),t)+4\bigg|\frac{d\gamma}{dt'}\bigg|^2_{g(t')}\bigg)dt'
\end{equation}
where $\gamma$ ranges over all curves connecting $x_1$ and $x_2$.

Combining with the integral lower bound (\ref{e: 317}) we can derive the pointwise Li-Yau type heat kernel lower bound.

\begin{corollary}
Under the assumption {\rm(\ref{Ricci: normalized})}, we have that
\begin{equation}\label{e: 324}
H(x,1-2t_0;y,1-t_0)\ge \frac{c(m)}{\vol_{g(1)}(B_{g(1)}(x_0,e^{-2}))},
\end{equation}
for any $x\in B_{g(1)}(x_0,e^{-2})$ and $y\in B_{g(1-t_0)}(x_0,e^{-5})$, where $t_0=t_0(m)\le\frac{1}{200}$ is the positive constant in Corollary \ref{Cheeger-Yau} and $c(m)$ is a positive constant depending only $m$.
\end{corollary}
\begin{proof}
Applying (\ref{e: 321}) with $t_1'=1-2t_0$ and $t_2'=1-\frac{3}{2}t_0$ we get
\begin{equation}\nonumber
H(x_2,1-\frac{3}{2}t_0;y,1-t_0)\le H(x_1,1-2t_0;y,1-t_0)\cdot C(m)\cdot\exp\bigg(\frac{L(x_1,1-2t_0;x_2,1-\frac{3}{2}t_0)}{t_0}\bigg).
\end{equation}
When $x_1,x_2\in B_{g(1)}(x_0,e^{-2})$, we can choose a minimal geodesic connecting them at time $t=1$, say $\gamma:[1-2t_0,1-\frac{3}{2}t_0]\rightarrow M$ with constant speed $\frac{2d_1(x_1,x_2)}{t_0}\le\frac{4}{et_0}$. Then $\gamma$ lies in the domain $B_{g(1)}(x_0,e^{-1})\subset B_0$, so by the uniform equivalence of the metrics $g(t)$ on $B_0$ we have
\begin{eqnarray}
L(x_1,1-2t_0;x_2,1-\frac{3}{2}t_0)&\le&\int_{1-2t_0}^{1-\frac{3}{2}t_0}\bigg(\frac{t_0^2}{4}\cdot R(\gamma(t'),t)+4\bigg|\frac{d\gamma}{dt'}\bigg|^2_{g(t')}\bigg)dt'\nonumber\\
&\le&\int_{1-2t_0}^{1-\frac{3}{2}t_0}\bigg(\frac{mt_0^2}{4}+4e^2\bigg|\frac{d\gamma}{dt'}\bigg|^2_{g(1)}\bigg)dt'\nonumber\\
&\le&C(m).\nonumber
\end{eqnarray}
In follows that
$$H(x_2,1-\frac{3}{2}t_0;y,1-t_0)\le C(m)\cdot H(x_1,1-2t_0;y;1-t_0)$$
for any $x_1,x_2\in B_{g(1)}(x_0,e^{-2})$. Thus, using that $B_{g(1-\frac{3}{2}t_0)}(x_0,e^{-3})\subset B_{g(1)}(x_0,e^{-2})$ we have the integral estimate
\begin{equation}
H(x_1,1-2t_0;y,1-t_0)\ge C(m)^{-1}\cdot\fint_{B_{g(1-\frac{3}{2}t_0)}(x_0,e^{-3})}H(x,1-\frac{3}{2}t_0;y,1-t_0)~dv_{g(1-\frac{3}{2}t_0)}(x)\nonumber.
\end{equation}
Now, by (\ref{e: 317}) we also have
\begin{equation}\nonumber
\int_{B_{g(1-\frac{3}{2}t_0)}(x_0,e^{-4})}H(x,1-\frac{3}{2}t_0;y,1-t_0)~dv_{g(1-\frac{3}{2}t_0)}(x)\ge\frac{1}{2}
\end{equation}
for any $y\in B_{g(1-t_0)}(x_0,e^{-5})$. So,
$$H(x_1,1-2t_0;y,1-t_0)\ge C(m)^{-1}\cdot\vol_{g(1-\frac{3}{2}t_0)}(B_{g(1-\frac{3}{2}t_0)}(x_0,e^{-3}))^{-1}$$
for any $y\in B_{g(1-t_0)}(x_0,e^{-5})$. The required heat kernel lower bound (\ref{e: 324}) follows from the metric equivalence (\ref{e: 301}) at times $t=1-\frac{3}{2}t_0$ and $t=1$ on $B_{g(1)}(x_0,1)$,
\begin{eqnarray}
\vol_{g(1-\frac{3}{2}t_0)}(B_{g(1-\frac{3}{2}t_0)}(x_0,e^{-3}))&\le& e^m\cdot\vol_{g(1)}(B_{g(1-\frac{3}{2}t_0)}(x_0,e^{-3}))\nonumber\\
&\le& e^m\cdot\vol_{g(1)}(B_{g(1)}(x_0,e^{-2})).\nonumber
\end{eqnarray}
\end{proof}

Next one can follow Perelman's argument \cite{Pe02} to improve the heat kernel lower bound in case $y$ admits a large distance from $x_0$. See also \cite{Wa17} for a similar argument.

As in \cite{Pe02}, for any $A\ge 1$, we let $\phi$ be a non-decreasing function of one variable, equal 1 on $(-\infty,\frac{1}{2}e^{-5})$, and rapidly increasing to infinity on $(\frac{1}{2}e^{-5},e^{-5})$, in such a way that
\begin{equation}
2(\phi')^2/\phi-\phi''\ge(t_0^{-1}A+2)\phi'-C(m,A)\cdot\phi
\end{equation}
for some constant $C(m,A)<\infty$, where $t_0=t_0(m)$ is the positive constant as in Corollary \ref{Cheeger-Yau}. Fix $x\in M$ and let
\begin{equation}
H(y,t)=H(x,1-2t_0;y,t),\,1-2t_0< t\le 1,
\end{equation}
be a solution to the forward heat equation. One can check easily that the function
$$h(y,t)=H(y,t)\cdot\phi\big(d_t(x_0,y)-t_0^{-1}A(t-(1-t_0))\big)$$
satisfies
\begin{equation}
\frac{\partial}{\partial t}h=\triangle h-2\phi^{-1}\langle\nabla\phi,\nabla h\rangle+\bigg(\phi'\big(\frac{\partial d_t}{\partial t}-\triangle d_t-t_0^{-1}A\big)-\phi''+2(\phi')^2/\phi\bigg)\cdot H.\nonumber
\end{equation}
By Lemma 8.3 (a) in \cite{Pe02} again, we have under the assumption (\ref{Ricci: normalized}), $\frac{\partial d_t}{\partial t}-\triangle d_t\ge-2$. Thus,
\begin{equation}
\frac{\partial}{\partial t}h\ge\triangle h-2\phi^{-1}\langle\nabla\phi,\nabla h\rangle-C(m,A)h.
\end{equation}
By maximum principle, $\min_{y\in M} \big(e^{C(m,A)t}h(y,t)\big)$ increases in $t$. In particular, since $\phi\ge 1$, we have
\begin{equation}\nonumber
\min_{B_{g(t)}(x_0,r(t))} H(\cdot,t)\le \min_{M}h(\cdot,t)\le e^{C(m,A)}\cdot\min_{M}h(\cdot,1)\le e^{C(m,A)}\cdot\min_{B_{g(1)}(x_0,A)}H(\cdot,1),
\end{equation}
where $r(t)=e^{-5}+t_0^{-1}A(t-(1-t_0))$ is a radius at time $t$ such that $h$ is infinite outside $B_{g(t)}(x_0,r(t))$. In particular,
\begin{equation}
\min_{B_{g(1)}(x_0,A)}H(\cdot,1)\ge e^{-C(m,A)}\cdot\min_{B_{g(1-t_0)}(x_0,e^{-5})} H(\cdot,1-t_0),
\end{equation}
in other words,
\begin{equation}
H(x,1-2t_0;y,1)\ge e^{-C(m,A)}\cdot \min_{z\in B_{g(1-t_0)}(x_0,e^{-5})} H(x,1-2t_0;z,1-t_0)
\end{equation}
for any $y\in B_{g(1)}(x_0,A)$. When $x\in B_{g(1)}(x_0,e^{-2})$, the right hand side admits a lower bound by (\ref{e: 324})
\begin{equation}\nonumber
\min_{z\in B_{g(1-t_0)}(x_0,e^{-5})} H(x,1-2t_0;z,1-t_0)\ge\frac{c(m)}{\vol_{g(1)}(B_{g(1)}(x_0,e^{-2}))}.
\end{equation}
So we conclude the following lemma.

\begin{lemma}
Under the assumption {\rm(\ref{Ricci: normalized})}, we have
\begin{equation}\label{e: 325}
H(x,1-2t_0;y,1)\ge \frac{c(m)\cdot e^{-C(m,A)}}{\vol_{g(1)}(B_{g(1)}(x_0,e^{-2}))}
\end{equation}
for any $x\in B_{g(1)}(x_0,e^{-2})$ and $y\in B_{g(1)}(x_0,A)$, where $t_0=t_0(m)\le\frac{1}{200}$ is the positive constant in Corollary \ref{Cheeger-Yau} and $c(m)$ is a positive constant depending only $m$.
\end{lemma}

\begin{remark}\label{constant}
According to the calculation in \cite{Wa17} the constant $C(m,A)$ can be chosen as $C(m)\cdot A^2$ for some constant $C(m)$.
\end{remark}

\begin{corollary}\label{conjugate equation estimate}
Let $v(x,t')$ be a nonnegative solution to the conjugate heat equation
\begin{equation}
-\frac{\partial}{\partial t'}v=\triangle v-Rv
\end{equation}
with initial $v\in C_0\big(B_{g(1)}(x_0,A)\big)$ satisfying
\begin{equation}
\int_Mv(x)dv_{g(1)}(x)=1.
\end{equation}
Then, under the assumption {\rm(\ref{Ricci: normalized})}, we have
\begin{equation}\label{e: 327}
v(x,1-2t_0)\ge \frac{c(m,A)}{\vol_{g(1)}(B_{g(1)}(x_0,e^{-2}))}
\end{equation}
for any $x\in B_{g(1-2t_0)}(x_0,e^{-3})$. In particular,
\begin{equation}\label{e: 328}
\int_{B_{g(1-2t_0)}(x_0,e^{-3})}v(x,1-2t_0)dv_{g(1-2t_0)}(x)\ge c(m,A)
\end{equation}
for some positive constant $c(m,A)$ depending on $m$ and $A$.
\end{corollary}
\begin{proof}
The solution $v$ has the formal representation
\begin{equation}
v(x,t')=\int_MH(x,t';y,1)\cdot v(y)~dv_{g(1)}(y)
\end{equation}
for any $t'<1$. Then, for any $x\in B_{g(1)}(x_0,e^{-2})$, (\ref{e: 325}) yields,
\begin{equation}\nonumber
v(x,1-2t_0)=\int_{B_{g(1)}(x_0,A)}H(x,1-2t_0;y,1)\cdot v(y)~dv_{g(1)}(y)\ge \frac{c(m,A)}{\vol_{g(1)}(B_{g(1)}(x_0,e^{-2}))}.
\end{equation}
Noticing that $B_{g(1-2t_0)}(x_0,e^{-3})\subset B_{g(1)}(x_0,e^{-2})$ we get the estimate (\ref{e: 327}). The last integral estimate (\ref{e: 328}) is a consequence of metric equivalence of $g(1-2t_0)$ and $g(1)$ on the metric ball $B_{g(1-2t_0)}(x_0,e^{-3})$ and the relative volume comparison at time $1$ on the metric ball $B_{g(1)}(x_0,e^{-2})\subset B_0$,
\begin{eqnarray}
\vol_{g(1-2t_0)}(B_{g(1-2t_0)}(x_0,e^{-3}))&\ge& e^{-m}\vol_{g(1)}(B_{g(1-2t_0)}(x_0,e^{-3}))\nonumber\\
&\ge& e^{-m}\vol_{g(1)}(B_{g(1)}(x_0,e^{-4}))\nonumber\\
&\ge& c(m)\cdot\vol_{g(1)}(B_{g(1)}(x_0,e^{-2})).\nonumber
\end{eqnarray}
\end{proof}

\section{Proof of Theorem \ref{relative noncollapsing}}

In this section we prove the main Theorem \ref{relative noncollapsing}.

Let $r_0>0$ be any constant. Let $g(t),0\le t\le r_0^2,$ be a solution to the Ricci flow on a compact $m$-manifold $M$. We assume that
\begin{equation}\label{Ricci: unnormalized}
|\Ric|\le r_0^{-2},\,\mbox{ on } B_{g(0)}(x_0,r_0)\times[0,r_0^2].
\end{equation}

Let $A\ge 1$ be a constant and $\Omega=B_{g(r_0^2)}(x_0,Ar_0)$ be a metric ball such that $\partial\Omega\neq\emptyset$. Define a family of parameters
$$\tau(t)=\tau_0+r_0^2-t$$
where $\tau_0$ is any positive constant less than $r_0^2$. Let $u(x)$ be a nonnegative minimizer of $\mu_\Omega(g(r_0^2),\tau_0)$ which extends continuously over $M$ such that $u$ vanishes outside of $\Omega$. Let $v(x,t)$, $0\le t\le r_0^2$, be the solution to the conjugate heat equation
\begin{equation}
-\frac{\partial}{\partial t}v=\triangle v-Rv,
\end{equation}
with initial value $v(r_0^2)=u^2$. By choosing $r=e^{-2}r_0$ in (\ref{relative volume ratio: 2}) we have
\begin{equation}\label{e 4.3}
\mu_{B}\big(g(t),\tau(t)\big)\le\mu_\Omega(g(r_0^2),\tau_0)
+200\cdot\bigg(\int_{B_{g(t)}(x_0,\frac{1}{2}e^{-1}r_0)}v(t)\bigg)^{-1}
\end{equation}
for any $(1-e^{-4})r_0^2\le t\le r_0^2$ and $0<\tau_0\le r^2= e^{-4}r_0^2$, where $B=B_{g(t)}(x_0,e^{-2}r_0)$.

Then, after a scaling of space-time and putting $t=(1-2t_0)r_0^2$ in (\ref{e: 328}) we obtain,
\begin{equation}
\int_{B_{g((1-2t_0)r_0^2)}(x_0,\frac{1}{2}e^{-1}r_0^2)}v(x,(1-2t_0)r_0^2)\ge c(m,A)
\end{equation}
for a positive constant $c(m,A)$. Here, $t_0=t_0(m)$ is the positive constant in Corollary  \ref{Cheeger-Yau}, which is so small that the time $(1-2t_0)r_0^2\ge(1-e^{-4})r_0^2$. Combining with (\ref{e 4.3}) we get
\begin{equation}\label{e: 510.5}
\mu_\Omega(g(r_0^2),\tau_0)\ge\mu_{B}\big(g\big((1-2t_0)r_0^2\big),\tau_0+2t_0 r_0^2\big)-C(m,A)
\end{equation}
for any $0<\tau_0\le r^2=e^{-4}r_0^2$, where $B=B_{g((1-2t_0)r_0^2)}(x_0,r)\subset B_0=B_{g(0)}(x_0,r_0)$. Next we shall apply Corollary \ref{entropy-volume: 1} to derive a lower bound of the entropy term on the right hand side of (\ref{e: 510.5}). To this purpose we choose any $0<\tau_0\le\frac{1}{10000}r_0^2$ such that
$$\tau_0+2t_0r_0^2\le\frac{1}{10000}r_0^2+\frac{1}{100}r_0^2\le r^2.$$
Noticing that $|\Ric|\le r_0^{-2}\le r^{-2}$ in $B$ we have by Corollary \ref{entropy-volume: 1},
\begin{equation}\label{510.75}
\mu_{B}\big(g\big((1-2t_0)r_0^2\big),\tau_0+2t_0r_0^2\big)\ge\log\frac{\vol_{g((1-2t_0)r_0^2)}(B_{g((1-2t_0)r_0^2)}(x_0,e^{-2}r_0))}{r_0^m}-C(m).
\end{equation}
By using the metric equivalence of $g((1-2t_0)r_0^2)$ and $g(0)$ on $B_{g((1-2t_0)r_0^2)}(x_0,e^{-2}r_0)$ the right hand side gives rise to the required the lower volume ratio bound. In fact, we have
\begin{eqnarray}
\vol_{g((1-2t_0)r_0^2)}(B_{g((1-2t_0)r_0^2)}(x_0,e^{-2}r_0))&\ge& e^{-m}\cdot\vol_{g(r_0^2)}(B_{g((1-2t_0)r_0^2)}(x_0,e^{-2}r_0))\nonumber\\
&\ge& e^{-m}\cdot\vol(B_{g(0)}(x_0,e^{-3}r_0))\label{e: 5512}\\
&\ge& C(m)^{-1}\cdot\vol(B_{g(0)}(x_0,r_0))\nonumber
\end{eqnarray}
where we used the Bishop-Gromov volume comparison in the last inequality. Summing up (\ref{e: 510.5})-(\ref{e: 5512}) yields the lower bound of the entropy
\begin{equation}\label{e: 511}
\inf_{0<\tau_0\le\frac{1}{10000}r_0^2}\mu_\Omega(g(r_0^2),\tau_0)\ge\log\frac{\vol_{g(0)}(B_{g(0)}(x_0,r_0))}{r_0^m}-C(m,A).
\end{equation}

Then we can use  Corollary \ref{entropy-volume: 2} to prove the upper bound $\mu_\Omega(g(r_0^2),\tau_0)$ in terms of the volume ratio of metric balls in $\Omega$. Actually, for any metric ball $B'=B_{g(r_0^2)}(x,r)\subset B_{g(r_0^2)}(x_0,Ar_0)$ we have the monotonicity of local entropy
\begin{equation}\label{e: 516}
\mu_\Omega(g(r_0^2),\tau_0)\le \mu_{B'}(g(r_0^2),\tau_0).
\end{equation}
On the other hand, by Corollary \ref{entropy-volume: 2}, if the scalar curvature at time $t=r_0^2$ satisfies
\begin{equation}\label{e: 517}
R(x,r_0^2)\le r^{-2},\,\forall x\in B',
\end{equation}
then
\begin{equation}\label{e: 518}
\min_{0<\tau_0\le r^2}\mu_{B'}(g(r_0^2),\tau_0)\le \log\frac{\vol_{g(r_0^2)}(B')}{r^m}+C(m).
\end{equation}
Hence, if the radius $r\le\frac{1}{100}r_0$, then by combining with (\ref{e: 511}), (\ref{e: 516}) and (\ref{e: 518}) we get
\begin{equation}
\log\frac{\vol_{g(r_0^2)}(B')}{r^m}\ge\log\frac{\vol_{g(0)}(B_{g(0)}(x_0,r_0))}{r_0^m}-C(m,A).
\end{equation}
This gives the desired estimate (\ref{relative volume}) in the main Theorem \ref{relative noncollapsing} when $r\le\frac{1}{100}r_0$. In the case when $ \frac{1}{100}r_0\le r\le r_0$ one can simply use $\vol_{g(r_0^2)}(B(x,r))\ge\vol_{g(r_0^2)}(B(x,\frac{1}{100}r_0))$ to get the estimate (\ref{relative volume}). The proof of Theorem \ref{relative noncollapsing} is complete.

\begin{remark}\label{constant 2}
By Remark \ref{constant} one can choose the constant $C(m,A)$ as
$$C(m,A)=C(m)\cdot e^{C(m)\cdot A^2}.$$
So the $\kappa$ in Theorem \ref{relative noncollapsing} can be chosen as
$$\kappa(m,A)=\exp\{-C(m)\cdot e^{C(m)\cdot A^2}\}.$$
\end{remark}

\section{K\"ahler-Ricci flow on smooth minimal models of Kodaira dimension one}

In this section, $X$ will be a K\"ahler manifold of complex dimension $n\ge 2$ with semiample canonical line bundle $K_X$ and Kodaira dimension one. For sufficiently large integer $\ell$, a basis of $H^0(X,K_X^\ell)$ defines a morphism $\pi:X\rightarrow X_{can}\subset\mathbb{C}P^{N_\ell}$ onto the canonical model $X_{can}$, a curve. Let $S=\{p_1,\cdots,p_q\}$ be the set of critical values of $\pi$. In \cite{SoTi07, SoTi12}, Song-Tian constructed the generalized K\"ahler-Einstein current $\omega_{\GKE}$ on $X_{can}$ which defines a smooth metric form on the regular set $X_{can}\backslash S$ such that
\begin{equation}
\Ric(\omega_{\GKE})=-\omega_{\GKE}+\omega_{\WP},\,\mbox{ on }X_{can}\backslash S,
\end{equation}
where $\omega_{\WP}$ is the Weil-Petersson form. It is pointed out by Y.S. Zhang in \cite{Zh17} that the metric completion of $(X_{can}\backslash S,\omega_{\GKE})$ is homeomorphic to $X_{can}$, so $\omega_{\GKE}$ defines a metric on $X_{can}$; see also \cite{ZhZh16} for the case when $X$ is a K\"ahler surface. We shall denote this metric space as $(X_{can},d_{\GKE})$. In this section we shall use the relative volume comparison of Ricci flow to study the geometric convergence of a K\"ahler-Ricci flow on $X$ to the generalized K\"ahler-Einstein space $(X_{can},d_{\GKE})$.

Let $\omega(t)$ be a K\"ahler-Ricci flow on $X$,
\begin{equation}\label{e: 610}
\frac{\partial}{\partial t}\omega=-\Ric-\omega.
\end{equation}
Due to the existence theorem of Tian-Zhang \cite{TiZh06}, the K\"ahler-Ricci flow has a global solution $\omega(t)$ for all $t\ge 0$. In the following $C_i$'s will be constants depending only on the initial metric $\omega(0)$.

Denote $\chi=\ell^{-1}\pi^*\omega_{FS}$ to be the pull-back of the Fubini-Study metric on $\mathbb{C}P^N$. Then we can write
\begin{equation}\label{e: 611}
\omega(t)=e^{-t}\omega(0)+(1-e^{-t})\chi+\sqrt{-1}\partial\bar{\partial}\varphi_t
\end{equation}
for a family of smooth real-valued functions $\varphi_t$ which satisfies a complex Monge-Amp\`ere equation
\begin{equation}
\frac{\partial}{\partial t}\varphi_t=\log\frac{\big(e^{-t}\omega(0)+(1-e^{-t})\chi+\sqrt{-1}\partial\bar{\partial}\varphi_t\big)^n}{e^{-(n-1)t}\Phi}-\varphi_t,
\end{equation}
with initial value $\varphi_0=0$. Here, $\Phi$ is a smooth volume form on $X$ such that $\sqrt{-1}\partial\bar{\partial}\log\Phi=\chi$. Then we recall some important estimates that shall be used later. In \cite{SoTi07, SoTi12} Song-Tian proved that $\omega(t)$ converges in the sense of currents to $\pi^*\omega_{\GKE}$; in \cite{SoTi16} they also proved that both $\varphi_t$ and $\frac{\partial\varphi_t}{\partial t}$ are uniformly bounded
\begin{equation}\label{e: 612}
|\varphi_t|+\big|\frac{\partial\varphi_t}{\partial t}\big|\le C_1,\,\forall t\in [0,\infty);
\end{equation}
it follows in particular that
\begin{equation}\label{e: 613}
 e^{-C_1-(n-1)t}\Phi\le\omega(t)^n\le e^{C_1-(n-1)t}\Phi;
\end{equation}
Song-Tian also proved in \cite{SoTi16} that the scalar curvature is uniformly bounded
\begin{equation}\label{e: 614}
|R(t)|\le C_2,\,\forall t\in[0,\infty).
\end{equation}
Tosatti-Weinkove-Yang proved in \cite{ToWeYa17} that $\omega(t)$ converges in the $C^0_{loc}$-topology on $X_{\reg}$, namely, for any compact subset $K\subset X_{can}\backslash S$
\begin{equation}\label{e:614.5}
\|\omega(t)-\pi^*\omega_{\GKE}\|_{C^0(K,\omega(0))}\rightarrow 0,\,\mbox{ as }t\rightarrow\infty.
\end{equation}

Now we give a proof of Theorem \ref{KRF: minimal model}.

\begin{proof}[Proof of Theorem \ref{KRF: minimal model}]
Assume as above. Let $S=\{p_1,\cdots,p_q\}$ be the set of critical values of $\pi$. For any positive number $0<\epsilon<1$ we define $U_\epsilon=\cup_{i=1}^NB_{\chi}(p_i,\epsilon)$ and let $\widetilde{U}_\epsilon=\pi^{-1}(U_\epsilon)$. Then, 
(\ref{e: 613}) implies
\begin{equation}\label{e: 616}
\limsup_{t\rightarrow\infty}\frac{\vol_{\omega(t)}(\widetilde{U}_\epsilon)}{\vol_{\omega(t)}(X)}=\delta(\epsilon)
\end{equation}
for some positive function $\delta$ satisfying $\lim_{\epsilon\rightarrow 0}\delta(\epsilon)=0$. Due to the estimate of $\omega_{\GKE}$ around singular points, cf. Theorem 1.1 of \cite{Zh17}, we have the trivial facts
\begin{equation}\label{e: 631}
d_{GH}\big((X_{can},d_{\GKE}),(X_{can}\backslash U_\epsilon,d_{\GKE})\big)=\delta'
\end{equation}
and, for any critical value $p_i$,
\begin{equation}\label{e: 631.5}
\diam(\partial B_\chi(p_i,\epsilon),d_{\GKE})\le\delta',\,\forall i,
\end{equation}
for some $\delta'=\delta'(\epsilon)$ which satisfies $\delta'\rightarrow 0$ as $\epsilon\rightarrow 0$.

In the following we use $d_t$ to denote the distance function of $\omega(t)$ on $X$. Let $\delta$ and $\delta'$ be any positive numbers less than 1. We may assume that $d_{\GKE}(p_i,p_j)\ge 4\delta$ for any $i\neq j$. Choose positive numbers $\epsilon\ll\delta$ and $T_1<\infty$ such that
\begin{equation}\label{e: 630}
\frac{\vol_{\omega(t)}(\widetilde{U}_\epsilon)}{\vol_{\omega(t)}(X)}\le\delta,\,\mbox{ for }\forall t\ge T_1.
\end{equation}
Fix one such $\epsilon$ from now on. As in the proof of Lemma 3.10 in \cite{ZhZh16}, applying the $C^0$ convergence on the compact set $\overline{X_{can}\backslash U_\epsilon}$, we can prove
\begin{equation}\label{e: 632}
d_{GH}\big((X_{can}\backslash U_\epsilon,d_{\GKE}),(X\backslash\widetilde{U}_\epsilon,d_t)\big)\le \delta,
\end{equation}
and, for any point $p_i$,
\begin{equation}\label{e: 632.5}
\diam\big(\partial (\pi^{-1}{B_\chi(p_i,\epsilon)}),d_t\big)\le2\delta',
\end{equation}
at any sufficiently large time $t$, say $t\ge T_2(\omega(0),\delta)$. In particular, if we denote by $D$ the diameter of $(X_{can},d_{\GKE})$, then by (\ref{e: 631}) and (\ref{e: 632}) we have
\begin{equation}
\diam(X\backslash\widetilde{U}_\epsilon,d_t)\le D +2,\,\mbox{ for }\forall t\ge T_2.
\end{equation}
It remains to estimate the Gromov-Hausdorff distance between $(X,d_t)$ and its subset $(X\backslash\widetilde{U}_\epsilon,d_t)$. We will use the relative volume comparison of Ricci flow to do this.

By assumption we have the uniform Ricci curvature bound on a domain $\pi^{-1}(U)$ where $U\subset X_{can}\backslash S$. Let $x_0\in U$ be a regular point and $0<r_0\le 1$ be a radius such that $d_{\GKE}(x_0,\partial U)<\frac{1}{2}r_0$. Let $\tilde{x}_0\in\pi^{-1}(x_0)$ be any inverse point. By assumption (\ref{Ricci: minimal model}), together with the $C^0$ convergence of $\omega(t)$ on $\overline{U}$, we may assume $r_0$ is so small that
\begin{equation}
|\Ric(x,t)|\le r_0^{-2},\,\mbox{ for any }x\in B_{\omega(t)}(\tilde{x}_0,r_0)
\end{equation}
and, by the scalar curvature estimate (\ref{e: 614}),
\begin{equation}
|R(x,t)|\le r_0^{-2},\,\mbox{ for any }x\in X,
\end{equation}
for any time $t\ge T_3$, a constant depending on $\omega(0)$. Moreover, we also have that $\pi^{-1}\big(B_{\omega_{\GKE}}(x_0,\frac{1}{2}r_0)\big)\subset B_{\omega(t)}(\tilde{x}_0,r_0)$ whenever $t\ge T_3$. According to Theorem 1.1 in \cite{Zh17}, the metric $\omega_{\GKE}$ defines the same topology as $\chi$ on the regular set. It implies that $B_\chi(x_0,r_0')\subset B_{\omega_{\GKE}}(x_0,\frac{1}{2}r_0)$ for some $r_0'>0$. So, by (\ref{e: 613}), the volume of $B_{\omega(t)}(\tilde{x}_0,r_0)$ can be estimated as follows
\begin{eqnarray}
\vol_{\omega(t)}(B_{\omega(t)}(\tilde{x}_0,r_0))
\ge\vol_{\omega(t)}(\pi^{-1}(B_\chi(x_0,r_0')))\ge e^{-C_1-(n-1)t}\int_{\pi^{-1}(B_\chi(x_0,r_0'))}\Phi.\nonumber
\end{eqnarray}
It follows that
\begin{equation}\label{e: 636}
\vol_{\omega(t)}\big(B_{\omega(t)}(\tilde{x}_0,r_0)\big)\ge C_3^{-1}\vol_{\omega(t)}(X)
\end{equation}
for some $C_3<\infty$ whenever $t\ge T_3$.

We next rescale the K\"ahler-Ricci flow $\omega(t)$ to the unnormalized one and then apply our relative volume comparison theorem. Fix any time $T\ge \max(T_1,T_2,T_3)$. Let $\tilde{t}(t)=\frac{1}{2}(e^{t-T}-1)$ for $t\ge T$ and $\tilde{\omega}(\tilde{t})=e^{t-T}\omega(t)=(1+2\tilde{t})\omega(t)$. Then $\tilde{\omega}$ is a solution to the Ricci flow
\begin{equation}
\frac{\partial}{\partial\tilde{t}}\tilde{\omega}=-2\widetilde{\Ric}
\end{equation}
where $\widetilde{\Ric}$ is the Ricci curvature form of $\tilde{\omega}$, with initial $\widetilde{\omega}(0)=\omega(T)$. Then
\begin{equation}
|\widetilde{\Ric}(x,\tilde{t})|\le r_0^{-2},\,\mbox{ for any }x\in B_{\tilde{\omega}(\tilde{t})}(\tilde{x}_0,r_0)
\end{equation}
and, the corresponding scalar curvature,
\begin{equation}
|\widetilde{R}(x,\tilde{t})|\le r_0^{-2},\,\mbox{ for any }x\in X,
\end{equation}
when $0\le \tilde{t}\le r_0^2$. Let $\tilde{d}_{\tilde{t}}$ be the distance function associated to $\tilde{\omega}(\tilde{t})$. Then,
$$\diam(X\backslash \widetilde{U}_\epsilon,\tilde{d}_{r_0^2})\le (1+2r_0^2)(D +2)\le 3(D+2).$$
Thus, if there exists a metric ball of radius $\rho\le 1$, say $B_{\tilde{\omega}(r_0^2)}(\tilde{x}_0,\rho)$, included in $\widetilde{U}_\epsilon$, such that $\partial B_{\tilde{\omega}(r_0^2)}(\tilde{x}_0,\rho)\cap\partial\widetilde{U}_\epsilon\neq\emptyset$, then the relative volume comparison theorem \ref{relative noncollapsing} implies
\begin{equation}\label{e: 640}
\frac{\vol_{\tilde{\omega}(r_0^2)}(B_{\tilde{\omega}(r_0^2)}(\tilde{x}_0,\rho))}
{\vol_{\tilde{\omega}(0)}(B_{\tilde{\omega}(0)}(\tilde{x}_0,r_0))}\ge \kappa(n,D)\cdot\frac{\rho^4}{r_0^4}
\end{equation}
where $\kappa(n,D)$ is a positive constant depending only on $n$ and $D$. On the other hand, we also have the upper bound,
$$\frac{\vol_{\tilde{\omega}(r_0^2)}(B_{\tilde{\omega}(r_0^2)}(\tilde{x}_0,\rho))}
{\vol_{\tilde{\omega}(0)}(B_{\tilde{\omega}(0)}(\tilde{x}_0,r_0))}
\le\frac{\vol_{\tilde{\omega}(r_0^2)}(\widetilde{U}_\epsilon)}{\vol_{\tilde{\omega}(r_0^2)}(X)}
\cdot\frac{\vol_{\tilde{\omega}(r_0^2)}(X)}{\vol_{\tilde{\omega}(0)}(X)}
\cdot\frac{\vol_{\tilde{\omega}(0)}(X)}
{\vol_{\tilde{\omega}(0)}(B_{\tilde{\omega}(0)}(\tilde{x}_0,r_0))}$$
where
$$\frac{\vol_{\tilde{\omega}(r_0^2)}(X)}{\vol_{\tilde{\omega}(0)}(X)}
=(1+2r_0^2)^m\cdot\frac{\vol_{\omega(T+t_0)}(X)}{\vol_{\omega(T)}(X)}\le C_{4}$$
where $t_0=\log(1+2r_0^2)$. Thus, together with (\ref{e: 630}) and (\ref{e: 636}), we have
$$\frac{\vol_{\tilde{\omega}(r_0^2)}(B_{\tilde{\omega}(r_0^2)}(\tilde{x}_0,\rho))}
{\vol_{\tilde{\omega}(0)}(B_{\tilde{\omega}(0)}(\tilde{x}_0,r_0))}
\le C_3\cdot C_4\cdot\delta.$$
Together with (\ref{e: 640}) it follows that
$$ \kappa(n,D)\cdot\frac{\rho^4}{r_0^4}\le C_3\cdot C_4\cdot\delta$$
which implies
\begin{equation}\nonumber
\rho\le C_5\cdot\delta^{1/4}
\end{equation}
for a constant $C_5$ independent of $t$. In particular, after rescaling, it follows that at time $t=T+\log(1+2r_0^2)$ where $T\ge \max(T_1,T_2,T_3)$,
$$d_{GH}\big((X\backslash\widetilde{U}_\epsilon,d_t),(X,d_t)\big)\le (1+2r_0^2)^{-1}\cdot C_5\cdot\delta^{1/4}+\sum_i\diam\big(\partial (\pi^{-1}{B_\chi(p_i,\epsilon)}),d_t\big).$$
By (\ref{e: 632.5}) we have
\begin{equation}
d_{GH}\big((X\backslash\widetilde{U}_\epsilon,d_t),(X,d_t)\big)\le C_5\cdot\delta^{1/4}+2q\delta'.
\end{equation}
Combining with (\ref{e: 631}) and (\ref{e: 632}) we finally get
\begin{equation}
d_{GH}\big((\Sigma,d_{\GKE}),(X,d_t)\big)\le C_6\cdot(\delta^{1/4}+\delta')
\end{equation}
whenever $t\ge\max\big(T_1,T_2,T_3)+\log(1+2r_0^2)$, where $C_6$ is a constant depending on the initial metric $\omega(0)$. Since $\delta$ and $\delta'$ are arbitrary, the Gromov-Hausdorff convergence is a consequence of this estimate.
\end{proof}

\section{Further discussions on Generalized K\"ahler-Einstein metric and K\"ahler-Ricci flow}

Let $X$ be an $n$-dimensional K\"ahler manifold with semi-ample canonical line bundle $K_X$. For any sufficiently large $\ell$, a basis of $H^0(X,K_X^\ell)$ defines a holomorphic Calabi-Yau fibration $\pi:X\rightarrow X_{can}\subset\mathbb{C}P^N$ onto its canonical model $X_{can}$. Let $S$ be the set of singular values of $\pi$, which is a subvariety of $X_{can}$. In \cite{SoTi07, SoTi12}, Song and the first named author constructed the unique generalized K\"ahler-Einstein current $\omega_{\GKE}$ on $X_{can}$. The current $\omega_{\GKE}$ is smooth on the regular set $X_{can}\backslash S$ and satisfies
\begin{equation}
\Ric(\omega_{\GKE})=-\omega_{\GKE}+\omega_{\WP},\,\mbox{ on }X_{can}\backslash S.
\end{equation}
As part of the AMMP, the Conjecture 6.3 in \cite{SoTi09} says that the K\"ahler-Ricci flow on $X$ ``\textit{converges to the unique generalized K\"ahler-Einstein metric $\omega_{\GKE}$ on $X_{can}$ in the sense of Gromov-Hausdorff}''. In the following, we give more discussions to this conjecture.

Let $d_{\GKE}$ be the induced length metric of $\omega_{\GKE}$ on $X_{can}\backslash S$.

\begin{definition}[Generalized K\"ahler-Einstein metric space]
Let
\begin{equation}
X_{\GKE}\,=\,\overline{(X_{can}\backslash S,d_{\GKE})}
\end{equation}
denote the metric completion of $(X_{can}\backslash S,d_{\GKE})$. We call $X_{\GKE}$ the generalized K\"aler-Einstein metric space associated to the Calabi-Yau fibration $\pi:X\rightarrow Y$. We will also use $d_{\GKE}$ to denote the extended metric on $X_{\GKE}$.
\end{definition}

According to the AMMP, the generalized K\"ahler-Einstein metric space should be identified with the canonical model of the manifold.

\begin{conjecture}\label{conj: id}
$X_{\GKE}$ is homeomorphic to $X_{can}$, in particular, $(X_{\GKE},d_{\GKE})$ is a compact metric space.
\end{conjecture}

Some special cases have been verified. When $X$ is a smooth minimal model of general type, the identification is proved in \cite{So14} by using the partial $C^0$ estimate. When the Kodaira dimension is 1, $X_{can}$ is a curve with isolated singularities and the identification is proved in \cite{Zh17}; see also \cite{ZhZh16} for the special case when $X$ is a K\"ahler surface.

\begin{remark}
The construction of generalized K\"ahler-Einstein metrics depends only on the Calabi-Yau fibration structure of the minimal model. The identification of $X_{\GKE}$ with $X_{can}$ does not depend on the specified K\"ahler-Ricci flow in AMMP.
\end{remark}

Now, let $(X,\omega(t))$, $t\in[0,\infty)$, be a K\"ahler-Ricci flow (\ref{e: 610}) on $X$. In \cite{SoTi07, SoTi12} it is proved that $\omega(t)$ converges in the sense of currents to $\pi^*\omega_{\GKE}$; it is also proved that (\ref{e: 612})-(\ref{e: 614}) hold uniformly in \cite{SoTi07, SoTi12, SoTi16}. Moreover, due to Tosatti-Weinkove-Yang \cite{ToWeYa17} (see also \cite{To15}), we have the local $C^0$ convergence of the metric tensor (\ref{e:614.5}), so the K\"ahler-Ricci flow collapses the regular fibres in a uniform way.

Motivated by the Ricci boundedness assumption in our theorems we also make the following conjecture. Notice that when the Kodaira dimension is strictly less than $n$ and the generic fibres are not necessarily tori, the sectional curvature along the K\"ahler-Ricci flow can never be bounded even along a regular fibre (cf. Theorem 1.3 \cite{ToZh16}).

\begin{conjecture}\label{conj: Ricci}
For any compact subset $K\subset X_{can}\backslash S$, the Ricci curvature admits a uniform bound
\begin{equation}
|\Ric|\le C,\,\mbox{ on }\pi^{-1}(K)\times[0,\infty).
\end{equation}
\end{conjecture}

If both Conjecture \ref{conj: id} and Conjecture \ref{conj: Ricci} can be affirmed, then our arguments in Section 5 can be adapted to bounding the diameter of the K\"ahler-Ricci flow, namely,
\begin{equation}
\diam(X,\omega(t))\le D,\,\forall t\ge 0.
\end{equation}
Moreover, applying the relative volume comparison Theorem \ref{relative noncollapsing}, we can also show the Gromov-Hausdorff convergence of the K\"ahler-Ricci flow. Another problem is how to identify the limit space. According to AMMP again we have the following conjecture.

\begin{conjecture}
The K\"ahler-Ricci flow $(X,\omega(t))$ converges in the Gromov-Hausdorff topology to $(X_{\GKE},d_{\GKE})$.
\end{conjecture}

\end{document}